\documentclass[11pt]{amsart}

\usepackage[english,frenchb]{babel}
\usepackage[T1]{fontenc}

\usepackage{latexsym,amssymb,amsmath}

\def\ffi{\varphi}
\def\eps{\varepsilon}
\def\dst{\displaystyle}

\def\C{{\mathbb{C}}}

\def\R{{\mathbb{R}}}

\newcommand{\norm}[1]{{\left\|{#1}\right\|}}

\newcommand{\abs}[1]{{\left|{#1}\right|}}
\newcommand{\scal}[1]{{\left\langle{#1}\right\rangle}}

\newtheorem{lemma}{Lemma}[section]
\newtheorem{prop}[lemma]{Proposition}
\newtheorem{theorem}[lemma]{Theorem}
\newtheorem{coro}[lemma]{Corollary}
\newtheorem{rema}[lemma]{Remark}

\begin{document}

\title[Fourier transforms and rearrangements]{On the Fourier transform of the symmetric decreasing rearrangements}

\author{Philippe Jaming}

\address{Universit\'e d'Orl\'eans\\
Facult\'e des Sciences\\
MAPMO - F\'ed\'eration Denis Poisson\\ BP 6759\\ F 45067 Orl\'eans Cedex 2\\
France}
\email{Philippe.Jaming@univ-orleans.fr}
%
%
%
\begin{abstract}
Inspired by work of Montgomery 
on Fourier series
and Donoho-Strak 
in signal processing, we investigate two families 
of rearrangement inequalities for the Fourier transform. More precisely,
we show that the $L^2$ behavior of a Fourier transform of a function
over a small set is controlled by the $L^2$ behavior of the Fourier transform of its symmetric 
decreasing rearrangement. In the $L^1$ case, the same is true if we further assume that the 
function has a support of finite measure.

As a byproduct, we also give a simple proof and an extension of a result of Lieb about 
the smoothness of a rearrangement. Finally, a straightforward application to solutions of the free 
Shr\"odinger equation is given.
\end{abstract}



\urladdr{http://jaming.nuxit.net/}


\maketitle

\section{Introduction}
The use of rearrangement techniques is a major tool for proving functional inequalities.
For instance, it has been used extensively for proving the boundedness
of the Fourier transform between weighted Lebesgue spaces ({\it see e.g.} \cite{JS1,JS2,BH} and 
the references therein). Let us mention that a weighted inequality for the Fourier transform 
was proved in
\cite{BH} with the help of a result of Jodeit-Torchinsky \cite{JT}
showing that an operator that is of type $(1,\infty)$ and of type $(2,2)$ satisfies
some rearrangement inequalities. 
%

Results mentioned so far deal with the rearrangement of Fourier transforms and not
with Fourier transforms of rearrangements. 
As the two operations are of course far from commuting, it is
thus not possible to deduce anything from them about the behavior of the Fourier transform of 
the rearrangement of a function. In that direction, a remarkable theorem is due to 
Lieb \cite[Lemma 4.1]{Li} that shows that the decreasing rearrangement preserves smoothness:

\begin{theorem}[Lieb for $s=1$ \cite{Li}, Donoho-Stark for $0<s<1$ \cite{DSu}]\label{th:lieb}
Let $d\geq 1$ be an integer and $0\leq s\leq 1$. Then there exists a constant $C_s$ such that,
for every $\ffi$ in the Sobolev space $H^s(\R^d)$
{\rm i.e.} 
$$
\norm{\ffi}_{H^s}:=\left(\int_{\R^d}|\widehat{\ffi}(\xi)|^2(1+|\xi|^2)^{s}\,\mathrm{d}\xi\right)^{1/2}<+\infty,
$$
$\bigl\||\ffi|^*\bigr\|_{H^s}\leq C_s\bigl\|\ffi\bigr\|_{H^s}$.
\end{theorem}

Further results relating smoothness and rearrangements may be found e.g. in papers by
A. Burchard \cite{Bu}, B. Kawohl \cite{Ka}. 
The best possible estimate
for second order derivatives was obtained by Cianchi \cite{Ci}. We will also show some sort of dual version of it,
namely that $\norm{\ffi}_{s}\leq C_s\norm{|\ffi|^*}_{s}$ if $s<0$.

In this paper, we are mainly dealing with a slightly different type of estimates. Namely, we will show that
the frequency content of the rearrangement of a function controls the frequency content of the function. We first show an ``$L^1$-result'' which may be stated as follows:

\begin{theorem}
\label{th1}
Let $S,\Omega>0$. Then there exists a constant $C=C(S,\Omega)$
such that, for every $\ffi\in L^1(\R^d)$ with support of finite measure $S$,
and for every $a,x\in\R^d$,
\begin{equation}
\label{eq:dsthm}
\abs{\int_{|\xi-a|\leq\Omega}\widehat{\ffi}(\xi)e^{2i\pi x\xi}\,\mathrm{d}\xi}\leq
C\int_{B(0,\Omega)}\widehat{|\ffi|^*}(\xi)\,\mathrm{d}\xi.
\end{equation}
\end{theorem}
This theorem is inspired and generalizes a result of Donoho and Stark which states that, for $d=1$,
$C(S,\Omega)=1$ provided $\Omega S$ is small enough, a result that we will generalize to higher dimension. One does not expect $C(S,\Omega)$ to be bounded when $S\to+\infty$ so that
the hypothesis on the support of $\ffi$ may not be lifted. Nevertheless, we show that the result
can be somewhat improved by taking Bochner-Riesz means.

We also show that in the ``$L^2$-case'' the situation is better:

\begin{theorem}\label{th:th1intro}
Let $d$ be an integer. Then there exists a constant $\kappa_d$ such that,
for every set $\Sigma\subset\R^{d}$ of finite positive measure and every $\ffi\in L^2(\R^d)$, 
$$
\int_\Sigma|\widehat{\ffi}(\xi)|^2\,\mathrm{d}\xi\leq 
\kappa_d\int_{B(0,\tau)}\abs{\widehat{|\ffi|^*}(\xi)}^2\,\mathrm{d}\xi,
$$
where $\tau$ is such that $|B(0,\tau)|=|\Sigma|$.
\end{theorem}

The similar estimate for Fourier series was previously obtained by Montgomery \cite{Mo}.
The result of Montgomery has also been generalized to Fourier transforms in a different way in
\cite{JS1,JS2} where rearrangements of Fourier transforms are considered rather than
Fourier transforms of rearrangements as in Theorem \ref{th:th1intro}.

Theorem \ref{th:th1intro} should be compared to Theorem \ref{th:lieb}.
Both theorems state that one can control one of the Fourier transforms
of $\ffi$ or of  $\ffi^*$ by the other one in a weighted $L^2$-sense.
Our results show that if the weight is ``small'' ({\it e.g.} the characteristic function of a set of 
finite measure), then $\widehat{\ffi^*}$ controls  $\widehat{\ffi}$. In Lieb's result, the weight is 
``big'' and the control goes the opposite way. We conjecture that this is also the case
when the weight is the characteristic function of the complementary of a set of finite 
measure, at least when $\ffi$ has a support of finite measure. We will show how 
this would imply
an optimal generalization to higher dimension of an uncertainty principle proved by F. Nazarov in
the one-dimensional case.

\medskip

Further, Theorem \ref{th:th1intro} may be interpreted in the following way: assume that
$\ffi$ has a big ``high-frequency component'', in the sense that
$\dst\int_\Sigma|\widehat{\ffi}(\xi)|^2\,\mathrm{d}\xi$ stays large
for all $\Sigma$ in a set of balls of radius $1$ and centered away from $0$,
then $\ffi^*$ must be big near to $0$. In other words, high-frequency oscillations 
are ``pushed'' to low-frequency oscillations by symmetrization. 

\bigskip

The paper is organized as follows. In the next section, we introduce the necessary notation
and prove a few simple preliminary lemmas. In Section 3, we prove Theorem \ref{th:lieb}
and its ``dual'' version.
The following section is devoted to the proof of Theorem \ref{th1}.
We pursue with the generalization of Montgomery's Theorem and present a conjecture related to Nazarov's Uncertainty 
Principle. We then give a direct illustration of the result in terms of solutions of the free Shr\"odinger equation. Finally, we explain how to extend our results to other symmetrizations. 

\section{Preliminaries and Notations}
\subsection{Generalities}\label{sec:gen}
Throughout this paper, $d$ will be an integer, $d\geq 1$. On $\R^d$, we denote by
$\scal{.,.}$ and $|.|$ the standard scalar product and the associated norm.
For $a\in\R^d$ and $r>0$, we denote by $B(a,r)$ the open ball centered at $a$ of radius $r$:
$B(a,r)=\{x\in\R^d\,:\ |x-a|<r\}$.

The Lebesgue measure on $\R^d$ is denoted $\mbox{d}x$ and we write
$|E|$ for the Lebesgue measure of a Borel set $E$. The various meanings
of $|A|$ should be clear from the context.
%

The Fourier transform of $\ffi\in L^1(\R^d)$ is defined by
$$
\widehat{\ffi}(\xi)=\int_{\R^d}\ffi(t)e^{-2i\pi\scal{t,\xi}}\,\mbox{d}t.
$$
This definition is extended from $L^1(\R^d)\cap L^2(\R^d)$ to $L^2(\R^d)$ in the standard way.
The inverse Fourier transform is denoted $\check\ffi$.
%

\subsection{Bessel functions and Fourier transforms}
Results in this section can be found in most books on Fourier
analysis, for instance \cite[Appendix B]{Gr}.

Let $\lambda$ be a real number with $\lambda>-1/2$. We define the
Bessel function $J_\lambda$ of order $\lambda$ on $(0,+\infty)$ by
its \emph{Poisson representation formula}
$$
J_\lambda(x)=\frac{x^\lambda}{2^\lambda
\Gamma\left(\lambda+\frac{1}{2}\right)\Gamma\left(\frac{1}{2}\right)}
\int_{-1}^1 (1-s^2)^\lambda\cos sx\frac{\mbox{d}s}{\sqrt{1-s^2}}.
$$
Let us define $\mathcal{J}_{-1/2}(x)=\cos x$ and for $\lambda>-1/2$,
$\mathcal{J}_\lambda(x):=\frac{J_\lambda(x)}{x^\lambda}$. Then
$\mathcal{J}_\lambda$ satisfies $|\mathcal{J}_\lambda(x)|\leq C_\lambda (1+|x|)^{-\lambda-1/2}$. 
It is also known that $\mathcal{J}_\lambda$, $\lambda>-1/2$, has
only positive real simple zeroes $(j_{\lambda,k})_{k\geq 1}$.


We will need the following well-known result:
\begin{equation}
\label{eq:fourboule}
\widehat{\chi_{B(0,1)}}(\xi)=\frac{J_{\frac{d}{2}}(2\pi|\xi|)}{|\xi|^{\frac{d}{2}}}
\end{equation}
from which we deduce that $|\widehat{\chi_{B(0,1)}}(\xi)|\leq C(1+|\xi|)^{-\frac{d+1}{2}}$
thus $\widehat{\chi_{B(0,1)}}\in L^p(\R^d)$ for all $p>\frac{2d}{d+1}$.

More generally, if we denote by $m_\alpha(x)=(1-|x|^2)^\alpha_+$, then
$$
\widehat{m_\alpha}(\xi)=\frac{\Gamma(\alpha+1)}{\pi^\alpha}
\frac{J_{\frac{d}{2}+\alpha}(2\pi|\xi|)}{|\xi|^{\frac{d}{2}+\alpha}}
=2^{\frac{d}{2}+\alpha}\pi^{\frac{d}{2}}\Gamma(\alpha+1)\mathcal{J}_{\frac{d}{2}+\alpha}(2\pi|\xi|).
$$

\subsection{Rearrangements}
For a Borel subset $E$ of $\R^d$ of finite measure,
we define its $d$-dimensional symmetric rearrangement $E^*$ to be the
open ball of $\R^d$ centered at the origin whose volume is that of $E$. Thus
$$
E^*=B(0,r)\quad\mbox{with}\quad
|B(0,1)|r^d=|E|.
$$

Let $\ffi$ be a measurable function on $\R^d$. We will say that $\ffi$ vanishes at infinity
if its level sets have finite measure\,: {\it i.e.}
the \emph{distribution function} of $\ffi$, $\mu_\ffi(\lambda)=|\{x\in\R^d\,:\ |\ffi(x)|>\lambda\}|$,
is finite for all $\lambda>0$. This is of course the case if $\ffi\in L^p$ for some $p\geq 1$.
We define the symmetric rearrangement $|\ffi|^*$
via the \emph{level-cake} representation\,:
$$
|\ffi|^*(x)
=\int_0^{+\infty}\chi_{\{y\in\R^d\,:\ |\ffi(y)|>\lambda\}^*}(x)\,\mbox{d}\lambda
$$
which has to be compared with the level-cake representation of $|\ffi|$\,:
$$
|\ffi(x)|=\int_0^{+\infty}\chi_{\{y\in\R^d\,:\ |\ffi(y)|>\lambda\}}(x)\,\mbox{d}\lambda.
$$

Rearrangements satisfy many useful properties: 

--- $|\alpha \ffi|^*=|\alpha||\ffi|^*$.

--- For a set $E$ of finite measure and for $\alpha>0$, 
$(\alpha E)^*=\alpha E^*$. Therefore, if we 
write $\ffi_\alpha(x)=\ffi(x/\alpha)$, then $|\ffi_\alpha|^*(x)=|\ffi|^*(x/\alpha)$.

--- $\ffi$ and $|\ffi|^*$ are equimeasurable, that is, for all $\lambda>0$
$\mu_\ffi(\lambda)=\mu_{|\ffi|^*}(\lambda)$. In particular, $\ffi$ and $|\ffi|^*$
have same $L^p$ norm for $1\leq p\leq\infty$.

--- The following theorems of Hardy, Littlewood and Riesz will be usefull 
({\it see e.g.} \cite[Chapter 3]{LL} for a proof):

\begin{lemma}[Hardy-Littlewood]\label{lem:rear2}
Let $\ffi,\psi$ be non-negative functions vanishing at infinity, then
$$
\int_{\R^d}\ffi(x)\psi(x)\,\mathrm{d}x\leq
\int_{\R^d}|\ffi|^*(x)|\psi|^*(x)\,\mathrm{d}x.
$$
\end{lemma}

\begin{lemma}[Riesz]\label{lem:riesz}
Let $\ffi,\psi,\chi$ be non-negative functions that vanish at infinity. Then
$$
\int_{\R^d}\int_{\R^d}\ffi(s)\psi(t)\chi(s-t)\,\mbox{d}s\,\mbox{d}t\leq
\int_{\R^d}\int_{\R^d}|\ffi|^*(s)|\psi|^*(t)|\chi|^*(s-t)\,\mbox{d}s\,\mbox{d}t
$$
\end{lemma}

We will also need the following result which is probably well known:

\begin{lemma}\label{lem:estrear}
If $\ffi\in L^p(\R^d)$, then 
$|\ffi|^*(t)\leq\dst \frac{\norm{\ffi}_p}{|B(0,1)|^{1/p}|t|^{d/p}}$ for all $t\in\R^d$.
\end{lemma}

\begin{proof}[Proof]
From Bienaym\'e Tchebichev we get $\mu_\ffi(\lambda)\leq\dst\frac{\norm{\ffi}_p^p}{\lambda^p}$
for all $\lambda>0$. But then,
$$
|B(0,1)|r^d:=|\{x\in\R\,:\ |\ffi|^*(x)>\lambda\}|=|\{y\in\R^d\,:\ |\ffi(y)|>\lambda\}|\leq
\frac{\norm{\ffi}_p^p}{\lambda^p}.
$$
Therefore, for $\zeta\in\R^d$ with $|\zeta|=1$, $|\ffi|^*\left(\left(\frac{\norm{\ffi}_p^p}{\lambda^p|B(0,1)|}\right)^{1/d}\zeta\right)\leq\lambda$. Inverting
this, we get $\dst|\ffi|^*(t)\leq\frac{\norm{\ffi}_p}{|B(0,1)|^{1/p}|t|^{d/p}}$ as claimed.
\end{proof}

In a similar way, one may prove that, for a function $\ffi$ on $\R^d$
that decays like $|\ffi(x)|\leq C(1+|x|)^{-\gamma}$, then
\begin{equation}
\label{eq:newrear}
|\ffi|^*(t)\leq C(1+|t|)^{-\gamma}.
\end{equation}

\section{Lieb's Theorem and its ``dual'' version}\label{sec:3}

Assume that either $\chi\in L^2(\R^d)$ and $\ffi\in L^2(\R^d)\cap L^1(\R^d)$
or vice versa $\ffi\in L^2(\R^d)$ and $\chi\in L^2(\R^d)\cap L^1(\R^d)$.
Then, 
\begin{eqnarray*}
\int_{\R^d}\widehat{\chi}(\xi)|\widehat{\ffi}(\xi)|^2\,\mbox{d}\xi&=&
\int_{\R^d}\widehat{\chi}(\xi)\widehat{\ffi}(\xi)\overline{\widehat{\ffi}(\xi)}\,\mbox{d}\xi\\
&=&\scal{\widehat{\chi}\widehat{\ffi},\widehat{\ffi}}
=\scal{\widehat{\chi*\ffi},\widehat{\ffi}}=\scal{\chi*\ffi,\ffi}
\end{eqnarray*}
with Parseval. The computation are justified by the fact that, as $\chi\in L^2(\R^d)$
and $\ffi\in L^1(\R^d)$ (or vice versa), $\chi*\ffi\in L^2(\R^d)$. As $\ffi\in L^2(\R^d)$, we may apply Fubini
to get
\begin{eqnarray*}
\int_{\R^d}\widehat{\chi}(\xi)|\widehat{\ffi}(\xi)|^2\,\mbox{d}\xi
&=&\int_{\R^d}\int_{\R^d}\chi(x-y)\ffi(x)\overline{\ffi(y)}\,\mbox{d}x\,\mbox{d}y\\
&\leq&\int_{\R^d}\int_{\R^d}|\chi|^*(x-y)|\ffi|^*(x)|\ffi|^*(y)\,\mbox{d}x\,\mbox{d}y
\end{eqnarray*}
with Riesz's Rearrangement Inequality (Lemma \ref{lem:riesz}).
Unwinding the previous computations, we thus obtain
\begin{equation}
\label{eq:weight}
\int_{\R^d}\widehat{\chi}(\xi)\bigl|\widehat{\ffi}(\xi)\bigr|^2\,\mbox{d}\xi
\leq\int_{\R^d}\widehat{|\chi|^*}(\xi)\bigl|\widehat{|\ffi|^*}(\xi)\bigr|^2\,\mbox{d}\xi.
\end{equation}

\begin{rema}
The hypothesis on $\ffi$ can not be totally lifted. 
For instance, one can not have the inequality with $\widehat{\chi}=\chi_S$
for large enough $S$ ({\it see} \cite{DS}). However, Theorem \ref{th:th1intro} gives a good substitute in that case.
\end{rema}

As an application of \eqref{eq:weight}, let us prove the following:

\begin{prop}
Let $s>0$, and let $\ffi\in L^2$. Then
$$
\int_{\R^d}(1+|\xi|^2)^{-s}|\widehat{\ffi}(\xi)|^2\,\mbox{d}\xi
\leq\int_{\R^d}(1+|\xi|^2)^{-s}|\widehat{\ffi^*}(\xi)|^2\,\mbox{d}\xi.
$$
\end{prop}

\begin{proof}[Proof]
Let us recall that, for each $s>0$, there exists a number $c_s$ such that, for every $\xi\in\R^d$,
$$
(1+|\xi|^2)^{-s}=c_s\int_{\R^d}|u|^{s-d}e^{-\pi(1+|\xi|^2)|u|^2}\,\mbox{d}u.
$$
Let us define $\chi_u(x)=c_s|u|^{s-2d}e^{-\pi|u|^2}e^{-\pi|\xi|^2/|u|^2}$. A simple computation then
shows that 
$\widehat{\chi_u}(\xi)=c_s|u|^{s-d}e^{-\pi(1+|\xi|^2)|u|^2}$. Applying \eqref{eq:weight}
to $\chi=\chi_u$, we obtain
$$
\int_{\R^d}c_s|u|^{s-d}e^{-\pi(1+|\xi|^2)|u|^2}|\widehat{\ffi}(\xi)|^2\,\mbox{d}\xi
\leq\int_{\R^d}c_s|u|^{s-d}e^{-\pi(1+|\xi|^2)|u|^2}|\widehat{\ffi^*}(\xi)|^2\,\mbox{d}\xi.
$$
Integrating over $u\in\R^d$ gives the result.
\end{proof}

As a second consequence of \eqref{eq:weight}, we prove Lieb Theorem and Donoho-Stark's extension of it.
The proof is directly inspired by Lieb's proof.

\begin{theorem}
Let $d\geq 1$ be an integer and $0\leq s\leq 1$. Then there exists a constant $C_s$
such that For $\ffi$ in the Sobolev space $H^s(\R^d)$, then
$\norm{|\ffi|^*}_{H^s}\leq C_s\norm{\ffi}_{H^s}$.
\end{theorem}

\begin{proof}[Proof]
The case $s=0$ is trivial:
\begin{equation}
\label{eq:trivff*}
\|\widehat{|\ffi|^*}\|_2=\||\ffi|^*\|_2=\norm{\ffi}_2=\|\widehat{\ffi}\|_2.
\end{equation}

If $0<s\le 1$, define $g_s(x)=e^{-\pi|x|^s}$ for $x\geq 0$. Then, $g_s$ is completely monotonic, that is, for every integer $k$,
$(-1)^k\partial^kg_s(x)\geq 0$ for $x>0$. According to a celebrated theorem of Bernstein
({\it see e.g.} \cite[Chapter 18, Section 4]{Fe} or \cite[page 161]{Wid}), $g_s$ is the Laplace transform
of a positive measure on $(0,+\infty)$. In particular, there exists a 
non-negative measure $\mu_s$ such that
$$
e^{-\pi|\xi|^{2s}}=\int_0^{+\infty}e^{-\pi t|\xi|^2}\,\mbox{d}\mu_s(t).
$$
From \eqref{eq:weight} applied to $\chi$ defined by $\chi(t)=t^{-d}e^{-\pi|x|^2/t}$, we obtain
$$
\int_{\R^d}e^{-\pi t|\xi|^2}\bigl|\widehat{\ffi}(\xi)\bigr|^2\,\mbox{d}\xi
\leq\int_{\R^d}e^{-\pi t|\xi|^2}\bigl|\widehat{|\ffi|^*}(\xi)\bigr|^2\,\mbox{d}\xi.
$$
Integrating with respect to $\mu_s(t)$, we thus obtain
$$
\int_{\R^d}e^{-\pi t|\xi|^{2s}}\bigl|\widehat{\ffi}(\xi)\bigr|^2\,\mbox{d}\xi
\leq\int_{\R^d}e^{-\pi t|\xi|^{2s}}\bigl|\widehat{|\ffi|^*}(\xi)\bigr|^2\,\mbox{d}\xi.
$$
With \eqref{eq:trivff*}, it follows that
$$
\int_{\R^d}\frac{1-e^{-\pi t|\xi|^{2s}}}{t}\bigl|\widehat{|\ffi|^*}(\xi)\bigr|^2\,\mbox{d}\xi
\leq\int_{\R^d}\frac{1-e^{-\pi t|\xi|^{2s}}}{t}\bigl|\widehat{\ffi}(\xi)\bigr|^2\,\mbox{d}\xi.
$$
Finally, $\dst\frac{1-e^{-\pi t|\xi|^{2s}}}{t}\to\pi|\xi|^{2s}$ as $t\to 0$ and
$\dst\frac{1-e^{-\pi t|\xi|^{2s}}}{t}\leq\pi|\xi|^{2s}$ so that Lebesgue's Lemma implies that
\begin{equation}
\label{eq:hs}
\int_{\R^d}|\xi|^{2s}\bigl|\widehat{|\ffi|^*}(\xi)\bigr|^2\,\mbox{d}\xi
\leq\int_{\R^d}|\xi|^{2s}\bigl|\widehat{\ffi}(\xi)\bigr|^2\,\mbox{d}\xi.
\end{equation}
A new appeal to \eqref{eq:trivff*} and to the fact that $1+|\xi|^{2s}\simeq (1+|\xi|^2)^s$ gives the result.
\end{proof}

Note that for $s=1$, $1+|\xi|^{2s}=(1+|\xi|^2)^s$ so that $C_1=1$.
Note also that the proof does not extend to $s>1$ as $g_s$ is no longer completely monotonic in that case.

\section{An extension of a result of Donoho and Stark}

\subsection{Higher dimensional generalization of Donoho and Stark's theorem}
We will start this section by giving a simple generalization of Donoho and Stark's result.
The main purpose of this is to explain the idea of the proof of the next theorem
which may appear a bit obscure and technical otherwise.

\begin{prop}\label{pro:ds}
Let $d\geq 1$ and $\alpha>-1/2$. Then there exists a constant $\vartheta=\vartheta(d,\alpha)$
such that, for every $\Omega,S>0$ with $\Omega S^{1/d}\leq \vartheta$
and for every $\ffi\in L^1(\R^d)$ with support of finite measure at most
$S$, we have
$$
\abs{\int_{\R^d}\Omega^{-d}m_\alpha(\xi/\Omega)\widehat{\ffi}(\xi)\,\mathrm{d}\xi}
\leq \int_{\R^d}\Omega^{-d}m_\alpha(\xi/\Omega) \widehat{|\ffi|^*}(\xi)\,\mathrm{d}\xi.
$$
\end{prop}

\begin{proof}[Proof] Let $R$ be such that the ball of radius $R$ has measure
$S$, $|B(0,R)|=S$. Plancherel's Formula shows that
$$
\Omega^{-d}\int_{\R^d}m_\alpha(\xi/\Omega)\widehat{\ffi}(\xi)\,\mbox{d}\xi
=2^{d/2+\alpha}\pi^{d/2}\Gamma(\alpha+1)
\int_{\R^d}\mathcal{J}_{d/2+\alpha}(2\pi\Omega|x|)\ffi(x)\,\mbox{d}x.
$$
It follows from Lemma \ref{lem:rear2} that
\begin{eqnarray*}
\abs{\int_{\R^d}\Omega^{-d}m_\alpha(\xi/\Omega)\widehat{\ffi}(\xi)\,\mbox{d}\xi}&\leq&\\
&&\hspace{-3cm}\leq\quad
2^{d/2+\alpha}\pi^{d/2}\Gamma(\alpha+1)
\int_{\R^d}|\mathcal{J}_{d/2+\alpha}(2\pi\Omega|x|)||\ffi(x)|\,\mbox{d}x\\
&&\hspace{-3cm}\leq\quad
2^{d/2+\alpha}\pi^{d/2}\Gamma(\alpha+1)
\int_{\R^d}|\mathcal{J}_{d/2+\alpha}|^*(2\pi\Omega|x|)|\ffi|^*(x)\,\mbox{d}x\\
&&\hspace{-3cm}=\quad
2^{d/2+\alpha}\pi^{d/2}\Gamma(\alpha+1)
\int_{B(0,R)}|\mathcal{J}_{d/2+\alpha}|^*(2\pi\Omega|x|)|\ffi|^*(x)\,\mbox{d}x.
\end{eqnarray*}

In order to complete the proof, it is enough to prove that $\mathcal{J}_{d/2+\alpha}$ has an ``absolute''
maximum at $0$\,:

\medskip

\noindent{\bf Fact.} {\sl There exists $\eps_0>0$ such that for all $t\in[0,\eps_0)$, 
$\mathcal{J}_{d/2+\alpha}(t)>\mathcal{J}_{d/2+\alpha}(\eps_0)$ and, for all $t\geq\eps_0$,
$|\mathcal{J}_{d/2+\alpha}(t)|\leq \mathcal{J}_{d/2+\alpha}(\eps_0)$.}

\medskip

We will postpone the proof of this fact to the end of this section. An immediate consequence of this fact
is that $|\mathcal{J}_{d/2+\alpha}|^*=\mathcal{J}_{d/2+\alpha}$ on $[0,\eps_0)$, thus,
if $2\pi\Omega|x|<\eps_0$ on $B(0,R)$ {\it i.e.} $\Omega S^{1/d}<\eps_0|B(0,1)|^{1/d}/2\pi$, then
\begin{eqnarray}
\abs{\int_{\R^d}\Omega^{-d}m_\alpha(\xi/\Omega)\widehat{\ffi}(\xi)\,\mbox{d}\xi}\nonumber\\
&&\hspace{-3cm}\leq\quad
2^{d/2+\alpha}\pi^{d/2}\Gamma(\alpha+1)
\int_{B(0,R)}\mathcal{J}_{d/2+\alpha}(2\pi\Omega|x|)|\ffi|^*(x)\,\mbox{d}x\label{eq:key}\\
&&\hspace{-3cm}=\quad
2^{d/2+\alpha}\pi^{d/2}\Gamma(\alpha+1)
\int_{\R^d}\mathcal{J}_{d/2+\alpha}(2\pi\Omega|x|)|\ffi|^*(x)\,\mbox{d}x\nonumber\\
&&\hspace{-3cm}=\quad\int_{\R^d}\Omega^{-d}m_\alpha(\xi/\Omega)\widehat{|\ffi|^*}(\xi)\,\mbox{d}\xi\nonumber
\end{eqnarray}
which completes the proof once we have proved the fact.
\end{proof}

Let us now give the proof of the fact (note that this is obvious for the sinc
function).

\begin{proof}[Proof of the fact]
First, $\mathcal{J}_{d/2+\alpha}$ is the Fourier transform
of a positive $L^1$ function and is therefore positive definite\,:
$$
\sum_{i,j=1,\ldots,n}\mathcal{J}_{d/2+\alpha}(|t_i-t_j|)\xi_i\overline{\xi_j}\geq 0
$$
for every integer $n$, every $t_1,\ldots,t_n\in\R^d$ and every $\xi_1,\ldots,\xi_n\in\C$.
As usual, by appropriately choosing the $\xi_i$'s and the $t_i$'s one gets that
$\mathcal{J}_{d/2+\alpha}(t)$ is maximal at $t=0$. As 
$|\mathcal{J}_{d/2+\alpha}(t)|\leq C/t^{d/2+1/2+\alpha}$, it is enough to show that, 
for every local maximum $t_i>0$ of $|\mathcal{J}_{d/2+\alpha}|$,
$|\mathcal{J}_{d/2+\alpha}(t_i)|<\mathcal{J}_{d/2+\alpha}(0)$.

To do so, note that from the positive definiteness of $\mathcal{J}_{d/2+\alpha}$, for every $t>0$, the matrix
$$
\begin{pmatrix}
\mathcal{J}_{d/2+\alpha}(0)&\mathcal{J}_{d/2+\alpha}(t_i)&\mathcal{J}_{d/2+\alpha}(t+t_i)\\
\mathcal{J}_{d/2+\alpha}(t_i)&\mathcal{J}_{d/2+\alpha}(0)&\mathcal{J}_{d/2+\alpha}(t)\\
\mathcal{J}_{d/2+\alpha}(t+t_i)&\mathcal{J}_{d/2+\alpha}(t)&\mathcal{J}_{d/2+\alpha}(0)
\end{pmatrix}
$$
is positive definite and thus has non-negative determinant. In particular, if $\mathcal{J}_{d/2+\alpha}(t_i)=\pm\mathcal{J}_{d/2+\alpha}(0)$,
$\bigl(\mathcal{J}_{d/2+\alpha}(t+t_i)\mp\mathcal{J}_{d/2+\alpha}(t)\bigr)^2\leq 0$. It follows that,
for all $t>0$, $\mathcal{J}_{d/2+\alpha}(t+t_i)=\pm\mathcal{J}_{d/2+\alpha}(t)$, which 
contradicts the estimate $|\mathcal{J}_{d/2+\alpha}(t)|\leq C/t^{d/2+1/2+\alpha}$.
\end{proof}

Finally, note that if, as in \cite{DS} we use Riesz's Inequality (Lemma \ref{lem:riesz})
instead of Hardy-Littlewood's Iniequality (Lemma \ref{lem:rear2}), we obtain the following:

\begin{prop}
Let $d\geq 1$ and $\alpha>-1/2$ and $\vartheta(d,\alpha)$
be the constant of Proposition \ref{pro:ds}. Let $\ffi\in L^2(\R^d)$ with support of finite measure $S$. 
If $\Omega S^{1/d}\leq \vartheta(d,\alpha)/2$, then
$$
\abs{\int_{\R^d}\Omega^{-d}m_\alpha(\xi/\Omega)|\widehat{\ffi}(\xi)|^2\,\mathrm{d}\xi}
\leq \int_{\R^d}\Omega^{-d}m_\alpha(\xi/\Omega)|\widehat{|\ffi|^*}(\xi)|^2\,\mathrm{d}\xi.
$$
\end{prop}

\subsection{The main theorem}
We are now in position to extend this result in the following way:

\begin{theorem}
\label{th:ds2}
Let $d$ be an integer and let $\beta\geq\frac{d}{2}-1$ and $\alpha> -\frac{1}{2}$.
Let $S,\Omega>0$. Set
$$
\psi(s)=\begin{cases}(1+s)^{\frac{1}{2}(d-2\alpha-1)}&\mbox{if }\alpha<\frac{d-1}{2} \\
\ln(1+s)&\mbox{if }\alpha=\frac{d-1}{2}  \\ 1&\mbox{if }\alpha>\frac{d-1}{2} \\\end{cases}.
$$
Then there exists a constant 
$C=C(d,\alpha,\beta)$ 
such that, for every $\ffi\in L^1(\R^d)$ with support of finite measure $S$,
for every $x,a\in\R^d$,
\begin{eqnarray}
\Omega^{-d}\abs{\int_{\R^d}\widehat{\ffi}(\xi)(1-|\xi-a|^2/\Omega^2)_+^\alpha
e^{2i\pi x\xi}\,\mathrm{d}\xi}&&\nonumber\\
&&\hspace{-4cm}\leq\quad
C\psi(\Omega^d S)
\int_{\R^d}\Omega^{-d}(1-|\xi|^2/\Omega^2)_+^\beta
\widehat{|\ffi|^*}(\xi)\,\mathrm{d}\xi.\label{eq:dsthm2}
\end{eqnarray}
\end{theorem}

\begin{rema}
The condition $\alpha>\frac{d-1}{2}$ that appears here for 
$\psi(s)$ to be constant is the same as the trivial bound for convergence of Bochner-Riesz means.
Actually both bounds only depend on bounds of appropriate Bessel functions. 
\end{rema}

\begin{proof}[Proof] It is enough to prove \eqref{eq:dsthm2} for $a=x=0$ and then apply it to 
$\ffi^{(a,x)}(t)=\ffi(t-x)e^{2i\pi at}$. Note that $|\ffi^{(a,x)}|^*=|\ffi|^*$ so that the
right hand side of \eqref{eq:dsthm} is unaffected by the change of $\ffi$ into $\ffi^{(a,b)}$.

We may further replace $\ffi$ by its dilate $\ffi(x/\Omega)$ so that, without loss of generality, 
we may assume that $\Omega=1$. More precisely, assume we are able to prove that
\begin{equation}
\label{eq:reduction}
\int_{\R^d}(1-|\xi|^2)^\alpha_+\widehat{\ffi}(\xi)\,\mbox{d}\xi\leq
\psi(S)\int_{\R^d}(1-|\xi|^2)^\beta_+\widehat{|\ffi|^*}(\xi)\,\mbox{d}\xi
\end{equation}
for every $S>0$ and every functions $\ffi$ with support of finite measure $S$.
We will then apply this to $\ffi_\Omega(x)=\ffi(x/\Omega)$ which has support of measure 
$\Omega^dS$. Note that
\begin{eqnarray*}
\int_{\R^d}(1-|\xi|^2)^\alpha_+\widehat{\ffi_\Omega}(\xi)\,\mbox{d}\xi
&=&\int_{\R^d}\Omega^d(1-|\xi|^2)^\alpha_+\widehat{\ffi}(\Omega\xi)\,\mbox{d}\xi\\
&=&\int_{\R^d}(1-|\xi|^2/\Omega^2)^\alpha_+\widehat{\ffi}(\xi)\,\mbox{d}\xi.
\end{eqnarray*}
On the other hand, $|\ffi_\Omega|^*(x)=|\ffi|^*(x/\Omega)$ therefore
$\widehat{|\ffi_\Omega|^*}(\xi)=\Omega^d\widehat{|\ffi|^*}(\Omega\xi)$ so that
$$
\int_{\R^d}(1-|\xi|^2)^\beta_+\widehat{|\ffi_\Omega|^*}(\xi)\,\mbox{d}\xi
=\int_{\R^d}(1-|\xi|^2/\Omega^2)^\beta_+\widehat{|\ffi|^*}(\xi)\,\mbox{d}\xi.
$$
It is thus enough to prove \eqref{eq:reduction}, {\it i.e.} to assume that
$\Omega=1$ in Theorem \ref{th:ds2}.
The beginning of the proof follows the lines of the proof of Proposition \ref{pro:ds}.

Using Parseval and Hardy-Littlewood's Symmetrization Lemma \ref{lem:rear2},
we obtain
\begin{eqnarray*}
\abs{\int_{\R^d}(1-|\xi|^2)_+^\alpha\widehat{\ffi}(\xi)\,\mathrm{d}\xi}
&=&\abs{\int_{\R^d}\widehat{m_\alpha}(t)\ffi(t)\,\mbox{d}t}\\
&\leq&
\int_{\R^d}|\widehat{m_\alpha}|^*(t)|\ffi|^*(t)\,\mbox{d}t.
\end{eqnarray*}
Recall that $\widehat{m_\alpha}(t)=C_{d,\alpha}|t|^{-d/2-\alpha}J_{d/2+\alpha}(2\pi|t|)$.
We will therefore write $|\widehat{m_\alpha}|^*(t)=\mathcal{J}_{d/2+\alpha}^*(|t|)$.

On the other hand
$$
\int_{\R^d}(1-|\xi|^2)_+^\beta\widehat{|\ffi|^*}(\xi)\,\mathrm{d}\xi
=\int_{\R^d}C_{d,\beta}|t|^{-d/2-\beta}J_{d/2+\beta}(2\pi|t|)
|\ffi|^*(t)\,\mbox{d}t.
$$
But, as $|\ffi|^*$ is a radial function that is ``radially decreasing'', we may 
write $|\ffi|^*(t)=\int\chi_{B(0,s)}(t)\,\mbox{d}\mu(s)$
where $\mu$ is a non-negative measure and $\chi_{B(0,s)}$ is the characteristic
function of the ball of center $0$ and radius $s$. It is thus enough to prove that,
for $s\leq S$,
\begin{eqnarray*}
\int_{\R^d}\mathcal{J}_{d/2+\alpha}^*(|t|)\chi_{B(0,s)}(t)\,\mbox{d}t&&\nonumber\\
&&\hspace{-3cm}\leq\quad\psi(S)
\int_{\R^d}C_{d,\beta}|t|^{-d/2-\beta}J_{d/2+\beta}(2\pi|t|)
\chi_{B(0,s)}(t)\,\mbox{d}t.
\end{eqnarray*}
Switching to polar coordinates, this is equivalent to proving that:
$$
\int_0^s\mathcal{J}_{d/2+\alpha}^*(r)r^{d-1}\,\mbox{d}r
\leq \psi(S)\int_0^s r^{d/2-\beta-1}J_{d/2+\beta}(2\pi r)
\,\mbox{d}t.
$$

Now, it is  well known that the graph of a Bessel function $J_\nu$ ($\nu>-1$)
consists of ``waves'' that are alternately positive and negative.
Moreover, the areas of these waves is strictly deceasing. That is, if we denote by
$j_{\nu,k}$ the $k$th zero of $J_\nu$, then
\begin{equation}
\label{eq:makai}
\int_{j_{\nu,k}}^{j_{\nu,k+1}}|J_\nu(r)|\,\mbox{d}r
\end{equation}
is strictly decreasing. As a consequence, we obtain that, for $\gamma\geq 0$,
\begin{eqnarray*}
\int_{j_{\nu,k}}^{j_{\nu,k+1}}r^{-\gamma}|J_\nu(r)|\,\mbox{d}r&>&
\int_{j_{\nu,k}}^{j_{\nu,k+1}}j_{\nu,k+1}^{-\gamma}|J_\nu(r)|\,\mbox{d}r\\
&>&\int_{j_{\nu,k+1}}^{j_{\nu,k+2}}j_{\nu,k+1}^{-\gamma}|J_\nu(r)|\,\mbox{d}r\\
&>&\int_{j_{\nu,k+1}}^{j_{\nu,k+2}}r^{-\gamma}|J_\nu(r)|\,\mbox{d}r.
\end{eqnarray*}
It follows that, for $0\leq\gamma<\nu+1$
$$
\int_{0}^{x}r^{-\gamma}|J_\nu(r)|\,\mbox{d}r\geq C_{\nu,\gamma}>0.
$$

On the other hand, from $|J_\nu(t)|\leq C_\nu (1+t)^{-1/2}$ and \eqref{eq:newrear}, we get
$\mathcal{J}_{d/2+\alpha}^*(|t|)\leq C_{d,\alpha}(1+|t|)^{-\frac{1}{2}(d+1+2\alpha)}$
thus
$$
\int_0^s\mathcal{J}_{d/2+\alpha}^*(r)r^{d-1}\,\mbox{d}r\leq \begin{cases}
C(1+s)^{\frac{1}{2}(d-2\alpha-1)}&\mbox{if }\alpha<\frac{d-1}{2}\\
C\ln(1+s)&\mbox{if }\alpha=\frac{d-1}{2}\\
C&\mbox{if }\alpha>\frac{d-1}{2}
\end{cases}
$$
which completes the proof.
\end{proof}

\begin{rema}
The behavior of Bessel functions used here is a classical result was proved originally by Cooke
\cite{Co1,Co2} using delicate estimates involving the Lommel functions and several properties
of Bessel functions. This is linked to the Gibbs Phenomena which is best seen
for $\nu=1/2$ (the case $d=1$, $\alpha=0$) for which $J_{1/2}(t)=\dst\frac{\sin t}{t}$.
The following estimates can then be found in many elementary courses in Fourier Analysis:
$$
\int_0^s\frac{\sin 2\pi t}{\pi t}\geq \begin{cases}s&\mbox{if }0\leq s\leq 1/2\\
2/5&\mbox{if } s\geq 1/2\end{cases}
$$
and $\dst\int_0^s\frac{\sin 2\pi t}{\pi t}\to\frac{1}{2}$ as $s\to+\infty$.

In \cite{Ma}, Makai proved \eqref{eq:makai} for $\nu>-1$ in a simpler way using a
differential equation approach of Sturm-Liouville type. A particularly simple proof of Cooke's
Theorem has been devised by Steinig in \cite{St}.

The range of $\gamma$'s can be extended to some negative values $\gamma>-\gamma(\nu)$
where $\gamma(\nu)$ is  defined in an implicite form for $-1<\nu<-1/2$, 
(Askey and Steinig \cite{AS}) and $\gamma(\nu)=-1/2$ for $\nu\ge -1/2$ (Gasper \cite{Ga}).
This allows to extend the theorem to $\beta\geq\frac{d-1}{2}$.
Further results may be found in \cite{MR}.
\end{rema}

The argument in the proof may be slightly modified to obtain the following:

\begin{coro}\label{cor:ds}
Let $d$ be an integer and let $\beta\geq\frac{d}{2}$ and $\alpha>-1/2$.
Let $S,\Omega>0$.
Then there exists a constant 
$C=C(d,\alpha,\beta)$ 
such that $\ffi\in L^1(\R^d)$ with support of finite measure $S$,
for every $a\in\R^d$,
\begin{eqnarray*}
\Omega^{-d}\int_{\R^d}|\widehat{\ffi}(\xi)|(1-|\xi-a|^2/\Omega^2)_+^\alpha\,\mathrm{d}\xi&&\\
&&\hspace{-3cm}\leq\quad
C(1+\Omega^d S)^{d/2}
\int_{\R^d}\Omega^{-d}(1-|\xi|^2/\Omega^2)_+^\beta
\widehat{|\ffi|^*}(\xi)\,\mathrm{d}\xi.
\end{eqnarray*}

More generally, if $\psi\in L^2(\R^d)\cap L^1(\R^d)$ then, for every $\ffi\in L^1(\R^d)$ with support of 
finite measure $S$,
\begin{eqnarray*}
\int_{\R^d}|\widehat{\ffi}(\xi)||\psi(\xi/\Omega)|\,\mathrm{d}\xi&&\\
&&\hspace{-3cm}\leq\ 
C\bigl(\norm{\psi}_1^{2/d}+\norm{\psi}_2^{2/d}\bigr)^{d/2}(1+\Omega^d S)^{d/2}
\int_{\R^d}(1-|\xi|^2/\Omega^2)_+^\beta
\widehat{|\ffi|^*}(\xi)\,\mathrm{d}\xi.
\end{eqnarray*}
\end{coro}

\begin{proof}[Sketch of proof]
Let us write $|\widehat{\ffi}|=\widehat{\ffi}$ then
\begin{eqnarray*}
\int_{\R^d}(1-|\xi|^2)_+^\alpha|\widehat{\ffi}(\xi)|\,\mathrm{d}\xi
&=&\int_{\R^d}e^{i\phi(\xi)}(1-|\xi|^2)_+^\alpha\widehat{\ffi}(\xi)\,\mathrm{d}\xi\\
&=&\abs{\int_{\R^d}[e^{i\phi(\cdot)}(1-|\cdot|^2)_+^\alpha]\widehat{\ \ }(t)\ffi(t)\,\mbox{d}t}\\
&\leq&
\int_{\R^d}|\mu_\alpha|^*(t)|\ffi|^*(t)\,\mbox{d}t
\end{eqnarray*}
where $\mu_\alpha$ is the Fourier transform of $e^{i\phi(\xi)}(1-|\xi|^2)_+^\alpha$.
One can not expect any better behavior of this function then to be in $L^2\cap L^\infty$
thus, with Lemma \ref{lem:estrear}, $|\mu_\alpha|^*(t)\leq C(1+|t|)^{-d/2}$. Therefore
$$
\int_0^s|\mu_\alpha|^*(r)r^{d-1}\,\mbox{d}r\leq C(1+s)^{d/2}
$$
(with the obvious abuse of notation). The remaining of the proof of the corollary
follows the line of the previous proof. 

The second statement is just a refinement of the previous one. We have to estimate
$|\widehat{\psi}|^*$ for which we use that $\|\widehat{\psi}\|_2=\norm{\psi}_2$
thus, from Lemma \ref{lem:estrear}, $|\widehat{\psi}|^*(t)\leq C\norm{\psi}_2|t|^{-d/2}$
and $\|\widehat{\psi}\|_\infty\leq\norm{\psi}_1$ thus $|\widehat{\psi}|^*(t)\leq C\norm{\psi}_1$.
Combining both estimates, we get 
$$
|\widehat{\psi}|^*(t)\leq C\bigl(\norm{\psi}_1^{2/d}+\norm{\psi}_2^{2/d}\bigr)^{d/2}(1+|t|)^{-d/2}.
$$
\end{proof}

\medskip

Finally, assume that $\ffi\in L^2$ with support of finite measure $S$.
Then $\ffi\in L^1$ and $|\widehat{\ffi}(\xi)|\leq\norm{\ffi}_1\leq S^{1/2}\norm{\ffi}_2$. Thus
\begin{eqnarray*}
\int_{B(0,\Omega)}|\widehat{\ffi}(\xi)|^2\,\mbox{d}\xi&\leq&
S^{1/2}\norm{\ffi}_2\int_{B(0,\Omega)}|\widehat{\ffi}(\xi)|\,\mbox{d}\xi\\
&\leq& \kappa(1+\Omega^d S)^{\frac{d}{2}}S^{1/2}\norm{\ffi}_2
\int_{B(0,\Omega)}\widehat{|\ffi|^*}(\xi)\,\mbox{d}\xi
\end{eqnarray*}
with Corollary \ref{cor:ds}. Using Cauchy-Schwarz, we thus get
$$
\int_{B(0,\Omega)}|\widehat{\ffi}(\xi)|^2\,\mbox{d}\xi\leq 
\kappa(1+\Omega^d S)^{\frac{d+1}{2}}\norm{\ffi}_2
\left(\int_{B(0,\Omega)}\abs{\widehat{|\ffi|^*}(\xi)}^2\,\mbox{d}\xi\right)^{1/2}.
$$
It turns out that this estimate may be improved. This is done in Section \ref{sec.mont} by
adapting a result originally proved for Fourier series in \cite{Mo}.

\section{The $L^2$ Theorem}\label{sec.mont}

\begin{theorem}
\label{th:mont}
Let $d\geq 1$ be an integer. Then there exists a constant $\kappa_d$ such that,
for every set $\Sigma\subset\R^d$ of finite positive measure and every $\ffi\in L^2(\R^d)$, 
\begin{equation}
\label{eq:mont}
\int_\Sigma|\widehat{\ffi}(\xi)|^2\,\mathrm{d}\xi\leq 
\kappa_d\int_{\Sigma^*}\abs{\widehat{|\ffi|^*}(\xi)}^2\,\mathrm{d}\xi.
\end{equation}
\end{theorem}

\begin{proof}[Proof] Let $\tau$ be defined by $|B(0,\tau)|=|\Sigma|$, that is
$B(0,\tau)=\Sigma^*$.

It is enough to prove the theorem for $\ffi\in L^1(\R^d)\cap L^2(\R^d)$.

For $\lambda>0$, let $\mathbb{D}_\ffi(\lambda)$ be the level set
$$
\mathbb{D}_\ffi(\lambda)=\{x\in\R^d\,:|\ffi(x)|>\lambda\}.
$$
Let us further choose $\bar\lambda$ to be the smallest non-negative real number $\lambda$
such that $|\mathbb{D}_\ffi(\lambda)|\leq|B(0,\tau^{-1})|$.
For simplicity of notation, we will  write $\mathbb{D}=\mathbb{D}_\ffi(\bar\lambda)$.
Let us further write $\widehat{\ffi}=f+g$ where
$$
f(\xi)=\int_{\mathbb{D}}\ffi(x)e^{-2i\pi\scal{x,\xi}}\,\mbox{d}x
\quad\mbox{and}\quad
g(\xi)=\int_{\R^d\setminus\mathbb{D}}\ffi(x)e^{-2i\pi\scal{x,\xi}}\,\mbox{d}x.
$$
First
\begin{eqnarray*}
\int_\Sigma|f(\xi)|^2\,\mbox{d}\xi&\leq&\norm{f}_\infty^2\int_\Sigma 1\,\mbox{d}\xi
\leq\left(\int_{\mathbb{D}}|\ffi(x)|\,\mbox{d}x\right)^2|\Sigma|\\
&\leq&|B(0,\tau)|\left(\int_{B(0,\tau^{-1})}|\ffi|^*(x)\,\mbox{d}x\right)^2.
\end{eqnarray*}
Further, using Parseval's Identity
$$
\int_\Sigma|g(\xi)|^2\,\mbox{d}\xi
\leq\int_{\R^d}|g(\xi)|^2\,\mbox{d}\xi=
\int_{\R^d\setminus\mathbb{D}}|\ffi(x)|^2\,\mbox{d}x
=\int_{|x|\geq\tau^{-1}}|\ffi|^*(x)^2\,\mbox{d}x.
$$
As $|\widehat{\ffi}|^2\leq 2|f|^2+2|g|^2$, we get
\begin{equation}
\label{eq:mont+}
\int_\Sigma|\widehat{\ffi}(\xi)|^2\,\mbox{d}\xi\leq
2|B(0,\tau)|\left(\int_{B(0,\tau^{-1})}|\ffi|^*(x)\,\mbox{d}x\right)^2
+2\int_{|x|\geq\tau^{-1}}|\ffi|^*(x)^2\,\mbox{d}x.
\end{equation}

On the other hand, let $K=\chi_{B(0,1)}*\chi_{B(0,1)}$ and note that
\begin{enumerate}
\renewcommand{\theenumi}{\roman{enumi}}
\item $0\leq K\leq K(0)=c_d:=|B(0,1)|=\frac{\pi^{d/2}}{\Gamma(d/2+1)}$, 
\item $K$ is supported in $B(0,2)$,
\item $K=\widehat{k^2}$ where
$\dst
k(x)=\widehat{\chi_{B(0,1)}}(x)=\frac{J_{d/2}(2\pi|x|)}{|x|^{d/2}}.
$
\end{enumerate}
Further, let $K_\tau(\xi)=\frac{1}{c_d}K(2\xi/\tau)$ and 
$k_\tau(x)=\left(\frac{\tau^d}{c_d2^d}\right)^{1/2}k(x\tau/2)$.
A simple computation then shows that
\begin{eqnarray}
\int_{B(0,\tau)}\abs{\widehat{|\ffi|^*}(\xi)}^2\,\mbox{d}\xi
&\geq&\int_{\R^d}K_\tau(\xi)\abs{\widehat{|\ffi|^*}(\xi)}^2\,\mbox{d}\xi\notag\\
&&\hspace{-2cm}=\quad
\int_{\R^d}\int_{\R^d}|\ffi|^*(x)|\ffi|^*(y)\int_{\R^d}K_\tau(x)e^{2i\pi\xi(x-y)}
\,\mbox{d}\xi\,\mbox{d}x\,\mbox{d}y\notag\\
&&\hspace{-2cm}=\quad
\int_{\R^d}\int_{\R^d}|\ffi|^*(x)|\ffi|^*(y)k_\tau^2(x-y)\,\mbox{d}x\,\mbox{d}y
:=I
\label{eq:montI}
\end{eqnarray}

Now, using the Poisson representation of Bessel functions,
\begin{eqnarray*}
k_\tau(u)&=&\left(\frac{\tau^d}{c_d2^d}\right)^{1/2}k(u\tau/2)
=\frac{J_{d/2}(\pi\tau|u|)}{\sqrt{c_d}|x|^{d/2}}\\
&=&\frac{2\sqrt{c_d}\tau^{d/2}}{(2\pi)\Gamma(1/2)}
\int_0^1(1-t^2)^{d-1/2}\cos(\pi|u|\tau t)\,\mbox{d}t.
\end{eqnarray*}
It is then obvious that, for $0\leq|u|\leq1/\tau$,
\begin{eqnarray*}
k_\tau(u)&\geq& k_\tau(1/\tau)=
\frac{2\sqrt{c_d}\tau^{d/2}}{(2\pi)^{d/2}\Gamma(1/2)}
\int_0^1(1-t^2)^{d-1/2}\cos(\pi t)\,\mbox{d}t\\
&=&(\upsilon_d|B(0,\tau)|)^{1/2}
\end{eqnarray*}
where
$$
\upsilon_d=\left(\frac{2}{(2\pi)^{d/2}\Gamma(1/2)}
\int_0^1(1-t^2)^{d-1/2}\cos(\pi t)\,\mbox{d}t\right)^2.
$$

We will now write $I\geq I_1+I_2$. For $I_1$,
we restrict the integration in \eqref{eq:montI} to
$x,y\in B(0,1/2\tau)$ and for $I_2$, the integration is restricted over
$|x|\geq 1/\tau$ and $|y-x|\leq1/tau$. 

Let us first estimate $I_1$:
\begin{eqnarray*}
I_1&\geq&
\upsilon_d|B(0,\tau)|\iint_{x,y\in B(0,1/2\tau)}|\ffi|^*(x)|\ffi|^*(y)\,\mbox{d}x\,\mbox{d}y\\
&\geq&\upsilon_d|B(0,\tau)|\left(\int_{|x|\leq\frac{1}{2\tau}}|\ffi|^*(x)\,\mbox{d}x\right)^2.
\end{eqnarray*}

From
$$
\int_{|x|\leq\frac{1}{2\tau}}|\ffi|^*(x)\,\mbox{d}x=
\int_{|x|\leq\frac{1}{\tau}}|\ffi|^*(x)\,\mbox{d}x
-\int_{\frac{1}{2\tau}\leq|x|\leq\frac{1}{\tau}}|\ffi|^*(x)\,\mbox{d}x
$$
we deduce that
$$
\int_{|x|\leq\frac{1}{2\tau}}|\ffi|^*(x)\,\mbox{d}x\geq
\frac{1}{2^d}\int_{|x|\leq\frac{1}{\tau}}|\ffi|^*(x)\,\mbox{d}x.
$$
We thus obtain the etimate
\begin{equation}
\label{est1}
I_1\geq\frac{\upsilon_d}{2^d}|B(0,\tau)|\left(\int_{B(0,1/\tau)}|\ffi|^*(x)\,\mbox{d}x\right)^2.
\end{equation}

On the other hand, if $|y|<|x|$, then $|\ffi|^*(y)\geq|\ffi|^*(x)$, so that
\begin{eqnarray*}
\int_{\{y\in\R^d\,:|y|<|x|,|x-y|\leq\frac{1}{\tau}\}}|\ffi|^*(y)k_\tau^2(x-y)\,\mbox{d}y\\
&&\hspace{-3cm}\geq\quad|\ffi|^*(x)\int_{\{y\in\R^d\,:|y|<|x|,|x-y|\leq\frac{1}{\tau}\}}k_\tau^2(x-y)\,\mbox{d}y.\\
\end{eqnarray*}
Using the properties of $k_\tau$ we then obtain
\begin{eqnarray*}
\int_{\{y\in\R^d\,:|y|<|x|,|x-y|\leq\frac{1}{\tau}\}}|\ffi|^*(y)k_\tau^2(x-y)\,\mbox{d}y\\
&&\hspace{-3cm}\geq\quad|\ffi|^*(x)\int_{|t|\leq\frac{1}{\tau},|x+t|<|x|}k_\tau^2(t)\,\mbox{d}t\\
&&\hspace{-3cm}\geq\quad|\ffi|^*(x)k_\tau^2(1/\tau)|B(0,1/\tau)\cap B(x,|x|)|\\
&&\hspace{-3cm}\geq\quad|\ffi|^*(x)\upsilon_d|B(0,\tau)||B(0,1/2\tau)|\\
&&\hspace{-3cm}=\quad\frac{\upsilon_d|B(0,1)|^2}{2^d}|\ffi|^*(x)
\end{eqnarray*}
provided $|x|\geq\frac{1}{\tau}$, since then $\{|u|\leq\tau^{-1},|x+u|<|x|\}\supset B(\frac{1}{2\tau} x/|x|,1/2\tau)$.

Therefore, as for $I_2$ we have restricted the integration in \eqref{eq:montI} to $(x,y)$'s such that
$|x|\geq\frac{1}{\tau}$, $|y|<|x|$ and $|x-y|\leq\frac{1}{\tau}$, we obtain
\begin{equation}
\label{est2}
I_2\geq \frac{\upsilon_d|B(0,1)|^2}{2^d}\int_{|x|\geq\frac{1}{\tau}}|\ffi|^*(x)^2\,\mbox{d}x.
\end{equation}

From \eqref{est1} and \eqref{est2} we get that
\begin{eqnarray*}
\int_{B(0,\tau)}\abs{\widehat{|\ffi|^*}(\xi)}^2\,\mbox{d}\xi&\geq&
\frac{\upsilon_d\min\bigl(1,|B(0,1)|^2\bigr)}{2^{d+1}}\times\\
&&\hspace{-3cm}\times
\left(2|B(0,\tau)|
\left(\int_{B(0,1/\tau)}|\ffi|^*(x)\,\mbox{d}x\right)^2
+2\int_{|x|\geq\frac{1}{\tau}}|\ffi|^*(x)^2\,\mbox{d}x\right)\\
&\geq&\frac{\upsilon_d\min\bigl(1,|B(0,1)|^2\bigr)}{2^{d+1}}\int_\Sigma\abs{\widehat{\ffi}(\xi)}^2\,\mbox{d}\xi
\end{eqnarray*}
in view of \eqref{eq:mont+}, as claimed.
\end{proof}

\begin{coro}
Let $d,\geq1$ be an integer and let $\kappa_d$ be the constant given by the previous theorem.
Let $\psi$ be a non-negative function on $\R^d$ that vanishes at infinity. Then for every $\ffi\in L^2(\R^d)$,
$$
\int_{\R^d}\psi(\xi)|\widehat{\ffi}(\xi)|^2\,\mathrm{d}\xi
\leq \kappa_d\int_{\R^d}|\psi|^*(\xi)\abs{\widehat{|\ffi|^*}(\xi)}^2\,\mathrm{d}\xi.
$$
\end{coro}

\begin{proof}[Proof]
For $\lambda>0$, Theorem \ref{th:mont} implies that
$$
\int_{\R^d}\chi_{\{\xi\,:\ \psi(\xi)>\lambda\}}|\widehat{\ffi}(\xi)|^2\,\mathrm{d}\xi
\leq \kappa_d\int_{\R^d}\chi_{\{\xi\,:\ \psi(\xi)>\lambda\}^*}\abs{\widehat{|\ffi|^*}(\xi)}^2\,\mathrm{d}\xi.
$$
Integrating this inequality over $\lambda>0$ and exchanging the order of integration gives the result.
This comes from the layer-cake representation for the left hand side and from the definition of $|\psi|^*$
for the right side.
\end{proof}

It seems natural to us to conjecture that if the characteristic function $\chi_\Sigma$ of $\Sigma$
is replaced by the characteristic function of its complement $\chi_{\Sigma^c}$ then the inequality in \eqref{eq:mont}
is reversed:

\medskip

\noindent{\bf Conjecture 1.}\\\
{\sl
Let $d\geq 1$ be an integer. Then there exists a constant $\kappa_d$ such that,
for every set $\Sigma\subset\R^d$ of finite positive measure and every $\ffi\in L^2(\R^d)$, 
\begin{equation}
\label{eq:montinv}
\int_{\R^d\setminus\Sigma^*}\bigl|\widehat{|\ffi|^*}(\xi)\bigr|^2\,\mathrm{d}\xi
\leq 
\kappa_d\int_{\R^d\setminus\Sigma}\bigl|\widehat{\ffi}(\xi)\bigr|^2\,\mathrm{d}\xi.
\end{equation}
}

Let us note that \eqref{eq:montinv} may be rewritten
$$
\int_{\R^d}\bigl|\widehat{|\ffi|^*}(\xi)\bigr|^2\,\mathrm{d}\xi
-\int_{\Sigma^*}\bigl|\widehat{|\ffi|^*}(\xi)\bigr|^2\,\mathrm{d}\xi
\leq 
\kappa_d\int_{\R^d}\bigl|\widehat{\ffi}(\xi)\bigr|^2\,\mathrm{d}\xi-
\int_{\Sigma}\bigl|\widehat{\ffi}(\xi)\bigr|^2\,\mathrm{d}\xi.
$$
As $\bigl\|\widehat{|\ffi|^*}\bigr\|_2=\bigl\|\widehat{\ffi}\bigr\|_2$,
this is equivalent to
$$
\int_{\Sigma}\bigl|\widehat{\ffi}(\xi)\bigr|^2\,\mathrm{d}\xi\leq
\left(1-\frac{1}{\kappa_d}\right)\int_{\R^d}\bigl|\widehat{\ffi}(\xi)\bigr|^2\,\mathrm{d}\xi
+\frac{1}{\kappa_d}\int_{\Sigma^*}\bigl|\widehat{|\ffi|^*}(\xi)\bigr|^2\,\mathrm{d}\xi.
$$

If this conjecture were true, then the following conjecture would follow:

\medskip

\noindent{\bf Conjecture 2.}\\\
{\sl
Let $d\geq 1$ be an integer. Then there exists a constant $C_d$
such that, if $S$ and $\Sigma$ are two sets of finite measure then, for every $\ffi\in L^2(\R^d)$,
$$
\norm{\ffi}_2^2\leq C_de^{C_d(|S||\Sigma|)^{1/d}}\left(\int_{\R^d\setminus S}\bigl|\ffi(x)\bigr|^2\,\mbox{d}x+
\int_{\R^d\setminus \Sigma}\bigl|\widehat{\ffi}(\xi)\bigr|^2\,\mbox{d}\xi\right).
$$
}
This conjecture has been proved in dimension $d=1$ by F. Nazarov \cite{Naz} and for $d\geq 2$
and either $S$ or $\Sigma$ convex by the author in \cite{Ja}. (The result was stated with a constant
of the form $C_de^{C_d\min(\omega(S)|\Sigma|^{1/d},\omega(\Sigma)|S|^{1/d})}$
where $\omega(S)$ --resp. $\omega(\Sigma)$-- is the mean width of $S$ --resp. $\Sigma$--
if this set is convex. But, it is a well known fact that $\omega(S)\leq C_d|S|^{1/d}$).

Let us now show how Conjecture 1 implies Conjecture 2. First, as is well known, it is equivalent to show that
$$
\int_{\Sigma}\bigl|\widehat{\ffi}(\xi)\bigr|^2\,\mbox{d}\xi
\leq C_de^{C_d(|S||\Sigma|)^{1/d}}
\int_{\R^d\setminus \Sigma}\bigl|\widehat{\ffi}(\xi)\bigr|^2\,\mbox{d}\xi
$$
for every $\ffi\in L^2$ with support in $S$. But, from Theorem \ref{th:mont},
$$
\int_{\Sigma}\bigl|\widehat{\ffi}(\xi)\bigr|^2\,\mbox{d}\xi
\leq\int_{\Sigma^*}\bigl|\widehat{|\ffi|^*}(\xi)\bigr|^2\,\mbox{d}\xi.
$$
Now $|\ffi|^*$ is supported in $S^*$, so that the particular case of Conjecture 2 that has already been proved in \cite{Ja} (and $|S^*|=|S|$, $|\Sigma^*|=|\Sigma|$)
implies that
$$
\int_{\Sigma}\bigl|\widehat{\ffi}(\xi)\bigr|^2\,\mbox{d}\xi
\leq C_de^{C_d(|S||\Sigma|)^{1/d}}\int_{\R^d\setminus\Sigma^*}\bigl|\widehat{|\ffi|^*}(\xi)\bigr|^2\,\mbox{d}\xi.
$$
Finally, once Conjecture 1 is established, this would imply that
$$
\int_{\Sigma}\bigl|\widehat{\ffi}(\xi)\bigr|^2\,\mbox{d}\xi
\leq \kappa_dC_de^{C_d(|S||\Sigma|)^{1/d}}\int_{\R^d\setminus\Sigma}\bigl|\widehat{\ffi}(\xi)\bigr|^2\,\mbox{d}\xi.
$$

Note that it is enough to establish Conjecture 1 for $\ffi$ with support $S$ of finite measure and that
one can allow for $\kappa_d=\kappa_d(S,\Sigma)=Ce^{C(|S||\Sigma|)^{1/d}}$.

\section{An application to the Free Shr\"odinger Equation}
Let us recall that the solution of the Free Shr\"odinger Equation
\begin{equation}
\label{eq:shr}
\begin{cases}\dst
i\partial_t v+\frac{1}{4\pi}\Delta_x^2 v=0\\
v(x,0)=v_0(x)\\
\end{cases}
\end{equation}
with initial data $v_0\in L^2(\R^d)$ has solution
\begin{eqnarray}
v(x,t)&=&\int_{\R^d}e^{-i\pi|\xi|^2t+2i\pi \scal{x,\xi}}\widehat{v_0}(\xi)\mbox{d}\xi\nonumber\\
&=&
\bigl(e^{-i\pi|\xi|^2t}\widehat{v_0}\bigr)\check\ (x)\label{solshr1}\\
&=&\frac{e^{i\pi|x|^2/t}}{(it)^{d/2}}\int_{\R^d}e^{i\pi|y|^2/t}v_0(y)e^{-2i\pi \scal{y,\xi}/t}\mbox{d}y\nonumber\\
&=&\frac{e^{i\pi|x|^2/t}}{(it)^{d/2}}\widehat{e^{i\pi|\cdot|^2/t}v_0}(x/t).
\label{solshr2}
\end{eqnarray}
We may thus apply Corollary \ref{cor:ds} and Theorem \ref{th:mont} to obtain a control of
$v(x,t)$ over sets of finite measure. This is stated in terms of $\widehat{|\widehat{v_0}|^*}$
for small time and of $\widehat{|v_0|^*}$ for large time. More precisely,
using respectively \eqref{solshr1} and \eqref{solshr2}, we obtain the following:

\begin{theorem}
Let $v$ be the solution of \eqref{eq:shr} and let $\Sigma$ be a set of finite measure
and let $\tau$ be defined by $|B(0,\tau)|=|\Sigma|$. Let $\beta\geq\frac{d-1}{2}$.

\begin{enumerate}
\item For every $t>0$,
$$
\int_\Sigma |v(x,t)|^2\,\mbox{d}x\leq C\int_{B(0,\tau)}\bigl|\widehat{|\widehat{v_0}|^*}(\xi)\bigr|^2\,\mbox{d}\xi.
$$
Moreover, if $\widehat{v_0}$ has support of finite measure $S$,
then
$$
\int_{\Sigma}|v(x,t)|\,\mbox{d}x\leq C|\Sigma|^{1/2}(1+|\Sigma|^{1/d})^{d/2}(1+S)
\int_{\R^d}(1-|\xi|^2)_+^\beta\widehat{|\widehat{v_0}|^*}(\xi)\,\mbox{d}\xi.
$$

\item For every $t>0$,
\begin{eqnarray}
\int_\Sigma |v(x,t)|^2\,\mbox{d}x
&\leq& Ct^{-d}\int_{B(0,\tau)}\bigl|\widehat{|v_0|^*}(\xi/t)\bigr|^2\,\mbox{d}\xi.\nonumber\\
&=&\int_{B(0,\tau/t)}\bigl|\widehat{|v_0|^*}(\xi)\bigr|^2\,\mbox{d}\xi.\label{eq:stri}
\end{eqnarray}
Moreover, if $v_0$ has support of finite measure $S$,
then
$$
\int_{\Sigma}|v(x,t)|\,\mbox{d}x\leq C|\Sigma|^{1/2}(1+|\Sigma|^{1/d}/t)^{d/2}(1+S)^{d/2}
\int_{\R^d}(1-|\xi|^2)_+^\beta\widehat{|v_0|^*}(\xi)\,\mbox{d}\xi.
$$
\end{enumerate}
Here $C$ is a constant that only depends on $d$ and on $\beta$.
\end{theorem}

For the second part of each statement, we have used $\psi=\chi_\Sigma$ and $\Omega=1$ in Corollary 
\ref{cor:ds} and a trivial change of variables.

Let us remark that, using H\"older and Hausdorff-Young, we have for $1\leq q\leq +\infty$, $\dst\frac{1}{q}+\frac{1}{q'}=1$
\begin{eqnarray*}
\int_{B(0,\tau/t)}\bigl|\widehat{|v_0|^*}(\xi)\bigr|^2\,\mbox{d}\xi
&=&\int_{\R^d}\chi_{B(0,\tau/t)}\bigl|\widehat{|v_0|^*}(\xi)\bigr|^2\,\mbox{d}\xi\\
&\leq&\Bigl\|\chi_{B(0,\tau/t)}\Bigr\|_q\Bigl\|\bigl|\widehat{|v_0|^*}(\xi)\bigr|^2\Bigr\|_{q'}\\
&=&|B(0,\tau/t)|^{1/q}\Bigl\|\bigl|\widehat{|v_0|^*}(\xi)\bigr|\Bigr\|_{2q'}^2\\
&\leq& \frac{|\Sigma|^{1/q}}{t^{d/q}}\bigl\||v_0|^*\bigr\|_{\frac{2q}{q+1}}^2.
\end{eqnarray*}
Setting $p=\dst\frac{2q}{q+1}\in[1,2]$, it follows from \eqref{eq:stri} that
$$
\int_\Sigma |v(x,t)|^2\,\mbox{d}x\leq C\frac{|\Sigma|^{2-p}}{t^{d(2-p)}}\bigl\|v_0\bigr\|_{p}^2.
$$
This estimate can also be obtained directly from the standard dispersive estimate
$$
\left(\int_{\R^d} |v(x,t)|^q\,\mbox{d}x\right)^{1/q}\leq Ct^{-\frac{d}{2}\left(1-\frac{2}{q}\right)}\norm{v_0}_{q'}
$$
($\dst\frac{1}{q}+\frac{1}{q'}=1$, $q\geq 2$)
and H\"older's inequality. The estimate \eqref{eq:stri} is slightly more precise and also shows that the case of radial initial data is somehow the worst case.

\section{Concluding remarks}

For the clarity of exposition, we have chosen to deal only with
the $d$-dimensional radial decreasing rearrangement of functions.
An alternative would have be to deal with the $1$-dimensional (resp. $k$-dimensional)
symmetric decreasing rearrangement
defined as follows: for a set $E\subset\R^d$ of finite measure, we define
$E^\star=[-a,a]$ where $a=|E|/2$
(resp. $E^\star=B(0,r)\subset\R^k$ with $|B(0,r)|_k=|E|_d$). 
One can then define the symmetric decreasing rearrangement
of functions through the layer cake representation:
$$
|\ffi|^\star(x)=\int_0^{+\infty}\chi_{\{y\in\R^d\,:\ |\ffi(y)|>\lambda\}^\star}(x)\,\mbox{d}\lambda.
$$
This rearrangement as similar properties to the one we considered so far. 
One may then adapt directly the proofs of Theorems \ref{th:ds2} and \ref{th:mont}
to obtain the following results:

\begin{theorem}
Let $d\geq1$ be an integer. 
\begin{enumerate}
\item Let $\alpha>-1/2$, $\beta\geq 0$. Let $S,\Omega>0$ and let $\psi$
be as in Theorem \ref{th:ds2}. Then there exists a constant 
$C=C(d,\alpha,\beta)$ 
such that, for every $\ffi\in L^1(\R^d)$ with support of finite measure $S$,
for every $x,a\in\R^d$,
\begin{eqnarray*}
\abs{\int_{\R^d}\widehat{\ffi}(\xi)(1-|\xi-a|^2/\Omega^2)_+^\alpha
e^{2i\pi x\xi}\,\mathrm{d}\xi}\\
&&\hspace{-4cm}\leq\quad
C\psi(\Omega^d S)
\int_{\R}(1-|\xi|^2/\Omega^2)_+^\beta
\widehat{|\ffi|^\star}(\xi)\,\mathrm{d}\xi.
\end{eqnarray*}

\item There exists a constant $\kappa_d$ such that, for every $\Sigma\in\R^d$ of finite measure
ad every $\ffi\in L^2(\R^d)$,
$$
\int_\Sigma\bigl|\widehat{\ffi}(\xi)\bigr|^2\,\mbox{d}\xi\leq\kappa_d
\int_{-|\Sigma|/2}^{|\Sigma|/2}\bigl|\widehat{|\ffi|^\star}(\xi)\bigr|^2\,\mbox{d}\xi.
$$
\end{enumerate}
\end{theorem}

\bibliographystyle{cdraifplain}

\end{document}